\documentclass{article}

\usepackage{amsfonts}
\usepackage{amssymb}
\usepackage{amsmath}
\usepackage{amsthm}
\usepackage{latexsym}
\usepackage{color}
\usepackage{amscd}
\usepackage{enumitem} 
\usepackage[margin=1.6in]{geometry}
\usepackage{esint} 
\usepackage{epsfig}

\allowdisplaybreaks

\usepackage{hyperref}
\usepackage{cleveref}
\numberwithin{equation}{section}
\newtheorem{theorem}{Theorem}[section]
\newtheorem{definition}{Definition}[section]
\newtheorem{lemma}{Lemma}[section]

\newtheorem{proposition}{Proposition}[section]

\newtheorem{remark}{Remark}[section]

\newcommand{\norm}[1]{\left\|{#1}\right\|}

\newcommand{\R}{\mathbb{R}}

\title{Uniqueness and stability for the Vlasov-Poisson system with spatial density in Orlicz spaces}
\author{Thomas Holding\footnote{University of Warwick, \emph{t.holding@warwick.ac.uk}} and Evelyne Miot\footnote{CNRS - Institut Fourier, Universit\'e Grenoble-Alpes, \emph{evelyne.miot@univ-grenoble-alpes.fr}}}

\begin{document}
\maketitle

\abstract{In this paper, we establish uniqueness of the solution of the Vlasov-Poisson system with spatial density belonging to a certain class of 
Orlicz spaces. This extends the uniqueness result of 
Loeper \cite{loeper} (which holds for density in $L^\infty$) and of the paper \cite{miot}. Uniqueness is a direct consequence of our main result, which 
provides a quantitative stability estimate for the Wasserstein distance between two weak solutions with spatial density in such Orlicz spaces, 
in the spirit of Dobrushin's proof of stability for mean-field PDEs. Our proofs are built on the second-order structure of the underlying characteristic system associated to the equation. }

\section{Introduction}
The purpose of this article is to study  uniqueness and stability issues for a class of weak solutions of the Vlasov-Poisson system in dimension $d=2$ or $d=3$, which reads:
\begin{equation}
\label{eq:Vlasov-Poisson}\begin{cases}
\displaystyle  \partial_t f+v\cdot \nabla_x f+E\cdot \nabla_v f=0,\quad (t,x,v)\in [0,T]\times \mathbb{R}^d\times \mathbb{R}^d \\
\displaystyle E(t,x)=\left(K\ast_x \rho\right)(t,x),\quad K(x)=\gamma \frac{x}{|x|^d}\\
\displaystyle \rho(t,x)=\int_{\mathbb{R}^d} f(t,x,v)\,dv.\end{cases} 
\end{equation} The system \eqref{eq:Vlasov-Poisson} describes the evolution of a microscopic density $f=f(t,x,v)$ of interacting particles, that are electric particles for $\gamma= 1$ (Coulombian interaction) 
or stars for $\gamma=-1$ (gravitational interaction). The function $\rho$ is called macroscopic (or spatial) density.

\medskip

Existence and uniqueness of  classical solutions of \eqref{eq:Vlasov-Poisson} defined on $[0,T]$ for all $T>0$ were 
established by Ukai and Okabe \cite{UO} for $d=2$ and by Pfaffelmoser  \cite{Pf} for $d=3$. Arsenev \cite{A} proved global existence of weak solutions with finite energy. Another
kind of global solutions, which propagate the velocity moments, was constructed by  Lions and Perthame \cite{lions-perthame}. We refer to the articles \cite{GJP, pallard}, and to references quoted 
therein, for further related results. On the other hand, part of the literature is devoted to determining sufficient conditions for uniqueness.  Loeper \cite{loeper} 
established uniqueness on $[0,T]$ in the class of weak solutions such that the spatial density $\rho$ is uniformly bounded:
\footnote{$\mathcal{M}_+(\mathbb{R}^d\times \mathbb{R}^d)$ denotes the space of bounded positive measures.}
\begin{equation}
\label{cond:loeper}
f\in C\left([0,T],\mathcal{M}_+(\mathbb{R}^{d}\times \mathbb{R}^{d})- w^\ast\right)\quad \text{and }\rho\in L^\infty\left([0,T],L^\infty(\mathbb{R}^d)\right).
\end{equation}This result was extended by
the second author in \cite{miot} to weak solutions satisfying 
\begin{equation}\label{condition:miot}
f\in C\left([0,T],\mathcal{M}_+(\mathbb{R}^{d}\times \mathbb{R}^{d})- w^\ast\right)\quad \text{and }\sup_{[0,T]} \sup_{p\geq 1} \frac{\|\rho(t)\|_{L^p}}{p}<+\infty.
\end{equation}

In \Cref{thm:main} below, we establish uniqueness of the solution with spatial density 
belonging to a certain class of exponential Orlicz spaces defined in \eqref{def:lux-2}. These spaces interpolate 
the functional spaces arising in \eqref{cond:loeper} and \eqref{condition:miot}. Our uniqueness result actually comes as a by-product of the main result of \Cref{thm:main},
which states a quantitative stability estimate involving the Wasserstein distance\footnote{See \Cref{def:W-1} hereafter of the Wasserstein distance.} 
between such weak solutions. We obtain this estimate in the spirit of  the method of Dobrushin \cite{dobrushin} to establish stability estimates for mean field PDE with Lipschitz convolution Kernels $K$.

\medskip

In the second part of this paper, we look for sufficient conditions on the initial data ensuring that any corresponding solution has spatial density belonging to the exponential Orlicz spaces defined in 
\eqref{def:lux-2} on $[0,T]$. In  
\Cref{prop:propagation-of-rho-bound} we prove that this holds for data with finite exponential velocity moment. 

\subsection{Main results}
\subsubsection{Preliminary definitions on Orlicz spaces and on the Wasserstein distance}
\paragraph{Orlicz spaces.} We begin by recalling some standard definitions related to Orlicz spaces. We refer the reader to e.g. \cite{Rao-Orlicz} for a more thorough exposition.
\begin{definition}[$N$-function]
We say that a function $\phi:[0,\infty)\to[0,\infty)$ is an $N$-function if it is continuous, convex with $\phi(\tau)>0$ for $\tau>0$ and satisfies both $\lim_{\tau\to0}\phi(\tau)/\tau=0$ and $\lim_{\tau\to\infty}\phi(\tau)/\tau=\infty$.
\end{definition}
\begin{definition}[Luxemburg norm]\label{def:lux} {Let $U$ be a domain of $\R^d$.}
	For an $N$-function $\phi$ we define the Luxemburg norm of a function $f$ defined on $U$ as
	\begin{equation}
	\norm{f}_{L_\phi(U)}=\inf\left\{\lambda>0:\int_{U}\phi(|f(x)|/\lambda)\,dx<1\right\}.
	\end{equation}
\end{definition}
\begin{remark}\label{rem:bounding-orlicz-norms}
If it holds for some constant $C'$ that
\begin{equation}
\int_{U}\phi(|f(x)|/\gamma)\,dx< C'
\end{equation}
then $\norm{f}_{L_\phi(U)}\le C\gamma$, where $C$ is an absolute constant depending only on $C'$.
\end{remark}
\begin{remark}\label{rem:equivalent-N-functions}
On bounded domains only the asymptotic behaviour as $\tau\to\infty$ of the $N$-function $\phi$ is important in defining the space $L_\phi$. In particular, if two $N$-functions $\phi,\tilde{\phi}$ have the same behaviour at infinity in the sense that there are $K,\tilde{K}>0$ such that $\phi(\tau)\le \tilde{\phi}(\tilde{K}\tau)$ and $\tilde{\phi}(\tau)\le \phi(K\tau)$ for all sufficiently large $\tau$, then the norms $\norm{\cdot}_{L_{\phi}(U)}$ and $\norm{\cdot}_{L_{\tilde{\phi}}(U)}$ are equivalent for any bounded domain $U\subseteq\mathbb{R}^d$.
\end{remark}
\begin{definition}[Complementary $N$-function]
For an $N$-function $\phi$ we define its complementary $N$-function $\bar{\phi}$ as
\begin{equation*}
\bar{\phi}(\tau)=\int^\tau_0a(s)\,ds
\end{equation*}
where $a$ is the right inverse of the right derivative of $\phi$.
\end{definition}
For $\alpha\in [1,+\infty)$ we let, for $\tau \geq 0$,
\begin{equation}\label{eq:phi_alpha}
\phi_\alpha(\tau)=\exp(\tau^\alpha)-1.
\end{equation}
The spaces $L_{\phi_\alpha}(\mathbb{R}^d)$ are exponential Orlicz spaces, and can be equivalently characterised as those functions $g$ which lie in $L^p$ for all $p\in[\alpha,\infty)$ and have the following norm finite:
\begin{equation}\label{def:lux-2}
\norm{g}_{\phi_\alpha}=\sup_{p\ge\alpha}p^{-1/\alpha}\norm{g}_{L^p(\mathbb{R}^d)},
\end{equation}
which is an equivalent norm to the Luxemburg norm $\norm{\cdot}_{L_{\phi_\alpha}(\R^d)}$. This equivalence is standard and can be verified by Taylor expansion of the 
exponential. Note that in the $\alpha\to\infty$ limiting case we obtain the function $\phi_{\infty}$ given by $\phi_\infty(\tau)=\infty$ if $\tau>1$ and $0$ otherwise. 
Although $\phi_\infty$ is not an $N$-function, we will use the convention that $L_{\phi_\infty}=L^\infty$. Therefore {with this convention} $L_{\phi_\alpha}$ indeed interpolates the functional spaces for $\rho$ that
are considered in \cite{loeper} (for $\alpha=+\infty$) and \cite{miot} (for $\alpha=1$).

\medskip

\begin{remark}
	When working with exponential Orlicz spaces, one has the choice between working with the 
	Luxemburg norm \eqref{def:lux} directly, or working with $L^p$ norms \emph{uniformly} in $p$ and using \eqref{def:lux-2}, as is done in \cite{miot} for $\alpha=1$. We take the former approach in this work.
\end{remark}

\paragraph{Transportation distances.}Let $n\geq 2$.
We let $\mathcal{M}_+(\mathbb{R}^n)$ denote the space of bounded positive measures on $\mathbb{R}^n$. 
\begin{definition}[Wasserstein distance]\label{def:W-1}
	For two measures $\mu,\nu\in\mathcal{M}_+(\mathbb{R}^n)$ with the same mass and finite first moments, we define the (Monge-Kantorovich-Rubenstein)-Wasserstein 
	distance $W_1(\mu,\nu)$ as\footnote{Here and throughout, $|v|$ denotes the euclidean norm of $v\in \R^n$.}
	\begin{equation*}
	 W_1(\mu,\nu)=\inf_{\pi\in \Pi(\mu,\nu)}\int_{\mathbb{R}^n\times\mathbb{R}^n}|x-y|\,d\pi(x,y),
	\end{equation*}
	where, here and throughout, $\Pi(\mu,\nu)$ denotes the set of \emph{couplings} 
	between $\mu$ and $\nu$, by which we mean measures in 
	$\mathcal{M}_+(\mathbb{R}^n\times\mathbb{R}^n)$ which have marginals $\mu$ and $\nu$ respectively.
\end{definition}
\begin{remark}
	The Wasserstein distance is usually defined on probability measures (i.e. elements of $\mathcal{M}_+$ with mass $1$) and {metrisizes}
	the weak* topology on the space of probability measures with finite first moment. In the case of 
	the extension to general bounded positive measures given above, it should be noted that the Wasserstein distance does not {metrisize} 
	the weak* topology on $\mathcal{M}_+$ with finite first moment. However, given any fixed mass $m$, the Wasserstein distance {metrisizes} the weak* topology on measures in $\mathcal{M}_+$ of mass $m$ with first moment finite.
\end{remark}
\subsubsection{Main results} 

We are now in position to state a quantitative estimate on the Wasserstein distance between two weak solutions of \eqref{eq:Vlasov-Poisson} with spatial density belonging to
some exponential Orlicz space: 
\begin{equation}\label{eq:assumption-on-rhos}
\rho_j\in L^\infty([0,T];L_{\phi_\alpha}(\mathbb{R}^d)) \text{ for }j=1,2\text{ and some }\alpha\in[1,\infty].
\end{equation}
\begin{theorem}\label{thm:main}Let $\varepsilon>0$ and $T>0$.
Let $f_1,f_2\in C([0,T], \mathcal{M}_+(\mathbb{R}^{d}\times \mathbb{R}^{d})-\textrm{w}^\ast)$ be two weak solutions of the Vlasov-Poisson 
system \eqref{eq:Vlasov-Poisson} with the same total mass such that \eqref{eq:assumption-on-rhos} holds. 
If {$(1+T)^{1+\varepsilon}W_1(f_{1}(0),f_2(0))<1/18$}, then we have the bound for $t\in[0,T^*]$:
\begin{itemize}
	\item 
If $\alpha=1$ then 
\begin{equation*}\label{eq:main-stability-estimate-1}
W_1(f_1(t),f_2(t))\le 
CW_1(f_{1}(0),f_2(0))^{\exp(-Ct)}\left(1+t|\ln W_1(f_1(0),f_2(0))|^2\right).\end{equation*}
\item 
If $\alpha>1$ then
\begin{equation*}\label{eq:main-stability-estimate}
\!\!\!W_1(f_1(t),f_2(t))\le 
CW_1(f_1(0),f_2(0))^{1/\gamma}\exp(Ct^\gamma)\left(1+t|\ln W_1(f_1(0),f_2(0))|^{1+(1/\alpha)}\right),
\end{equation*}
where
\begin{equation*}\label{eq:gamma}
\gamma = \frac{2}{1-(1/\alpha)}\in[2,+\infty),
\end{equation*}
\end{itemize}
and where $T^*$ satisfies the lower bound
\begin{equation*}
 T^\ast \geq  \begin{cases}\displaystyle C'\ln |\ln W_1(f_{1}(0),f_2(0))|-C'^{-1}\quad \text{when }\alpha=1,\\ 
 \displaystyle C' \gamma |\ln W_1(f_{1}(0),f_2(0))|^{1/\gamma}-C'^{-1} \quad \text{when }\alpha \neq 1.\end{cases} 
\end{equation*}(we set $T^\ast=T$ if the right hand side is larger than $T$).
 The constants $C$ and $C'$ depend only upon the norms of $\rho_1,\rho_2$ in \eqref{eq:assumption-on-rhos} and on $\varepsilon$.
\end{theorem}
\begin{remark}
The bound is stated in a way that is easy to understand for \emph{large} $t$ and is suboptimal near $t=0$. 
In particular the bound does not converge to $W_1(f_1(0),f_2(0))$ as $t\to0$\footnote{They do, of course, converge to zero as $W_1(f_1(0),f_2(0))\to0$.}. Such a bound could be obtained by a careful analysis of the proofs, but we do not present this here.
\end{remark}
\begin{remark}As will be clear in the proof of \Cref{thm:main}, the time $T^\ast$ essentially corresponds to the first time at which the right hand side 
becomes larger or equal to $1$.
\end{remark}

\begin{remark}\label{remark:case-infinity} For $\alpha=+\infty$, the estimate of Theorem \ref{thm:main} reads
\begin{equation*}
W_1(f_1(t),f_2(t))\le 
CW_1(f_1(0),f_2(0))^{1/2}\exp(Ct^2)\left(1+t|\ln W_1(f_1(0),f_2(0))|\right), 
\end{equation*}which is valid up to times of order $ |\ln W_1(f_{1}(0),f_2(0))|^{1/2}$.
\end{remark}

In \cite{dobrushin}, Dobrushin considered the stability of measure-valued solutions of first order {mean-field PDE} with Lipschitz convolution Kernels $K$ and obtained the 
inequality $$W_1(f_1(t),f_2(t))\le W_1(f_1(0),f_2(0))\exp\left(Ct\|\nabla K\|_{L^\infty}\right).$$ The same estimate was derived by Moussa and Sueur \cite{moussa-sueur} for a mixed first/second order PDE. Hauray and Jabin \cite{hauray-jabin} handled the case of more singular Kernels, see also the recent work by Lazarovici and Pickl \cite{lazarovici2015mean} on cut-off kernels and the references quoted therein. 

In the present situation, we are able to address the case of the singular convolution Kernel 
$K=\gamma x/|x|^d$ because, in contrast with the works mentioned above, the solutions have some additional regularity -
the macroscopic density belongs to $L_{\phi_\alpha}$. Nevertheless, as a consequence of the singularity of $K$, the growth of $W_1(f_1(t),f_2(t))$ in \Cref{thm:main} is not 
linearly bounded in terms of  $W_1(f_1(0),f_2(0))$.

We mention that although stability estimates are not explicitly done in \cite{loeper}, the computations therein involve a log-Lipschitz Gr\"onwall estimate and
would yield the inequality
	\begin{equation}\label{ineq:loeper}
	W_2(f_1(t),f_2(t))\le CW_2(f_1(0),f_2(0))^{\exp(-Ct)},
	\end{equation}with $W_2$ denoting 
	\begin{equation*}
	W_2(\mu,\nu)=\left(\inf_{\pi\in \Pi(\mu,\nu)}\int_{\mathbb{R}^d\times\mathbb{R}^d}|x-y|^2\,d\pi(x,y)\right)^{1/2},
	\end{equation*}
so the $W_2$-Wasserstein distance grows in time roughly like an exponential tower $e^{e^{ct}}$. Therefore the estimate of \Cref{thm:main} setting $\alpha=+\infty$, which corresponds to the regularity considered in 
	\cite{loeper}, improves this to stretched exponential growth of the form $e^{ct^2}$. This improvement is due to the second-order structure of the characteristic
	system \eqref{eq:characteristics} of ODE associated to the Vlasov-Poisson system, which was already exploited in the proof of uniqueness in \cite{miot}.

	Finally, we would like to point out that the same technique of exploiting the second-order structure can be applied to general measure solutions (with no regularity assumption on the spatial density), and allows the Dobrushin estimate to be improved slightly from \emph{Lipschitz} kernels to \emph{log${}^2$-Lipschitz} kernels:
\begin{theorem}\label{thm:log2-lipschitz}
Let the convolution kernel $K$ be bounded and satisfy the log${}^2$-Lipschitz property:
\begin{equation}\label{eq:log2-lipschitz-assumption}
|K(x)-K(y)|\le C|x-y|(\ln|x-y|)^2\text{ for all }|x-y|\le 1/9.
\end{equation}
Then the Vlasov-Poisson system \eqref{eq:Vlasov-Poisson} has a unique solution such that $f$ belongs to $C([0,T];\mathcal{M}_+(\mathbb{R}^d\times\mathbb{R}^d)-\textrm{w}^\ast)$ 
for any initial datum in $\mathcal{M}_+(\mathbb{R}^d\times\mathbb{R}^d)$. Moreover it obeys the stability estimate, for any two solutions 
$f_1,f_2\in C([0,T];\mathcal{M}_+(\mathbb{R}^d\times\mathbb{R}^d)-\textrm{w}^\ast)$ with the same mass and satisfying $(1+T)W_1(f_{1}(0),f_2(0))<1/9$,
\begin{equation*}
W_1(f_1(t),f_2(t))\le CW_1(f_{1}(0),f_2(0))^{\exp(-Ct)}(1+t|\ln W_1(f_1(0),f_2(0))|^2).
\end{equation*}
{which holds for times $t\in[0,T^*]$ with $T^*$ defined analogously to \Cref{thm:main}.}
\end{theorem}
We remark that the conventional improvement of the Dobrushin estimate by replacing the Gr\"onwall inequality with a log-Lipschitz inequality only allows one to treat log-Lipschitz kernels $K$, rather than the slightly weaker assumption \eqref{eq:log2-lipschitz-assumption}.

\medskip

In the second part of our analysis, we seek for initial data $f_0$ for which the macroscopic density indeed belongs to some exponential Orlicz space.
\begin{proposition}\label{prop:propagation-of-rho-bound}
Let $f_0\in L^\infty(\mathbb{R}^{d}\times \mathbb{R}^{d})$ be such that
\begin{equation*}
\int_{\R^d\times \R^d} f_0(x,v)e^{c\<v^{d\alpha}}\,dxdv<\infty,
\end{equation*}
for some $\alpha\in[1,\infty)$ and $c>0$,{where $\<v=\sqrt{1+|v|^2}$}. For $T>0$, let $f\in C([0,T],\mathcal{M}_+(\mathbb{R}^{d}\times \mathbb{R}^{d})- w^\ast)\cap L^\infty([0,T],L^1\cap L^\infty(\mathbb{R}^{d}\times \mathbb{R}^{d}))$ 
be any solution to \eqref{eq:Vlasov-Poisson}, with this initial datum, provided by  \cite[Theo. 1]{lions-perthame}\footnote{The existence of such a 
solution is ensured by \cite[Theo. 1]{lions-perthame} because 
$f_0$ has finite velocity moments of sufficiently large order: $
 \int |v|^m f_0(x,v)\,dx\,dv<+\infty
$ for some $m>d^2-d$. This is proved in 
\cite{lions-perthame} for $d=3$. The case $d=2$ 
is a straightforward adaptation of the case $d=3$.}.
Then it satisfies
\begin{equation*}
\sup_{t\in[0,T]}\norm{\rho(t)}_{L_{\phi_\alpha}}\le C<\infty.
\end{equation*}
In particular, this solution satisfies the uniqueness criterion of \Cref{thm:main}.
\end{proposition}
We remark that setting $\alpha=1$, we retrieve as a particular case the condition obtained in
\cite[Theo. 1.2]{miot} to ensure that \eqref{condition:miot} holds.

\medskip

The plan of the remainder of the paper is as follows. In Section \ref{sec:2} we prove Theorem \ref{thm:main}. We first establish in Lemma 
\ref{lem:harmonic-analysis} a log-Lipschitz like estimate for the 
force field $E=K\ast \rho$ associated to a function $\rho$ satisfying \eqref{eq:assumption-on-rhos}. Then, 
we introduce in \eqref{def:distance} a notion of distance between two solutions in terms of the characteristics 
defined in \eqref{eq:characteristics}, which controls the Wasserstein distance (see \eqref{carac:wasserstein}). 
This quantity was used in the original proof of Dobrushin and also
in \cite{miot}, while the proof of \cite{loeper} uses a slightly different version. Applying similar arguments as in \cite{dobrushin}, we derive a second-order differential 
inequality for this distance, which eventually leads to Theorem \ref{thm:main}. {In \Cref{subsec:proof-of-log2-lipschitz} we show how to adapt this technique to prove \Cref{thm:log2-lipschitz}.} Finally, the last Section \ref{sec:3} is devoted to the proof of \Cref{prop:propagation-of-rho-bound}.

\section{Proof of Theorem \ref{thm:main}}\label{sec:2}

\subsection{An estimate for the Newton kernel}
To prove \Cref{thm:main} we have need of the following lemma on the Newton kernel. Note that the complementary $N$-functions of the $\phi_{\alpha}$ behave asymptotically (see \Cref{rem:equivalent-N-functions}) like
\begin{equation}\label{eq:asymptotic-of-bar-phi-alpha}
\bar{\phi}_{\alpha}(\tau)\sim\tau\ln(\tau)^{1/\alpha}\text{ as }\tau\to\infty.
\end{equation}
Recall that Orlicz spaces obey a form of H\"older's inequality (see e.g. \cite{Rao-Orlicz})
\begin{equation*}
\left|\int_{\R^d} fg\,dx\right|\le C\norm{f}_{L_\phi}\norm{g}_{L_{\bar{\phi}}}
\end{equation*}
for the constant $C>1$.

Given $\alpha\in[1,\infty]$ we define the constant $\beta\in[1,2]$ by
\begin{equation}\label{eq:beta}
\beta = \frac1\alpha+1.
\end{equation}
In particular, note that $\beta=1$ for $\alpha=+\infty$ and $\beta=2$ as $\alpha=1$.
\begin{lemma}\label{lem:harmonic-analysis}
Let $\alpha\in[1,\infty)$, then there exists $C=C(\alpha)>0$ such that for all $g\in L_{\phi_\alpha}\cap L^1$ we have the estimate
\begin{equation}
\int_{\mathbb{R}^d} \left|K(x-z)-K(y-z)\right||g(z)|\,dz\le C(\norm{g}_{L_{\phi_\alpha}}+\norm{g}_{L^1})\psi_\alpha(|x-y|),
\end{equation}
where $\psi$ is defined by
\begin{equation}\label{eq:psi-alpha}
\psi_\alpha(\tau)=\begin{cases}\displaystyle \tau|\ln(\tau)|^\beta,\quad \text{for }\tau\in \left[0,\frac{1}{9}\right],\\
\displaystyle \frac{1}{9}\ln(9)^\beta,\quad \text{for }\tau\geq \frac{1}{9}\end{cases}
\end{equation}
and where $\beta$ is defined by \eqref{eq:beta}. 
\end{lemma}

\begin{remark}\label{rem:alpha-infinity-case} In the case $\alpha=+\infty$, namely for $g\in L^1\cap L^\infty$ the estimate of Lemma \ref{lem:harmonic-analysis} 
is standard, see e.g. \cite[Lemma 8.1]{majda-bertozzi} for the case $d=2$: we have
\begin{equation*}
\int_{\mathbb{R}^d} \left|K(x-z)-K(y-z)\right||g(z)|\,dz\le C(\|g\|_{L^1\cap L^\infty})|x-y|\left(1+ |\ln | x-y||\right).
\end{equation*}
\end{remark}

\begin{remark} For the case $\alpha=1$, the following variant of Lemma \ref{lem:harmonic-analysis} was obtained in \cite[Lemma 2.2]{miot}: for all $x,y\in \mathbb{R}^d$ with $|x-y|$ sufficiently small,  
\begin{equation*}
\int_{\mathbb{R}^d} \left|K(x-z)-K(y-z)\right||g(z)|\,dz\le C\, p\, (\|g\|_{L^1}+\|g\|_{L^p})|x-y| ^{1-d/p},\quad \forall p> d.
\end{equation*}
In particular, recalling \eqref{def:lux-2} for $\alpha=1$, this yields
\begin{equation*}
\int_{\mathbb{R}^d} \left|K(x-z)-K(y-z)\right||g(z)|\,dz\le C  \|g\|_{L_{\phi_1}}|x-y| \left(p^2| x-y|^{-d/p}\right),\quad \forall p> d,
\end{equation*}
so setting $p= |\ln | x-y||$ we retrieve the estimate of Lemma \ref{lem:harmonic-analysis}. In fact one can also prove the other cases via this method. Nevertheless, 
we give a direct proof of \Cref{lem:harmonic-analysis} below for completeness.
\end{remark} 

\begin{proof}
We set $$\delta=\frac{1}{d}\left(1+\frac{1}{\alpha}\right)=\frac{\beta}{d}\in \left(\frac{1}{d},\frac{2}{d}\right].$$
By standard estimates using H\"older's inequality (see e.g. \cite{livrejaune}) it is well-known that, fixing some $p_0>d$,
\begin{equation*}
\int_{\mathbb{R}^d} \left|K(x-z)-K(y-z)\right||g(z)|\,dz\le C(\norm{g}_{L^{p_0}}+\norm{g}_{L^1})\leq C(\norm{g}_{L_{\phi_\alpha}}+\norm{g}_{L^1}).
\end{equation*} Hence, in view of the form of $\psi_\alpha$, letting $R=|x-y|$ we may assume without loss of generality that 
$R\le 1/9$.  We introduce $A=(x+y)/2$. Since $R |\ln R|^\delta<1$, we may split the integral as follows:
	\begin{align*}
		\int_{\mathbb{R}^d} &\left|K(x-z)-K(y-z)\right||g(z)|\,dz\\
		&=\int_{\mathbb{R}^d\setminus B(A,{|\ln R|^{-\delta}})} \left|K(x-z)-K(y-z)\right||g(z)|\,dz\\
		&+\int_{B(A,{|\ln R|^{-\delta}})\setminus B(A,R)} \left|K(x-z)-K(y-z)\right||g(z)|\,dz+\int_{B(A,R)} \left|K(x-z)-K(y-z)\right||g(z)|\,dz\\
		&=I_1+I_2+I_3.
	\end{align*}
	For $I_1$ we apply the mean value theorem to obtain the bound
	\begin{equation*}\begin{split}
	I_1&\le C\norm{g}_{L^1}\,R\,\sup_{u\in[x,y],z\in \mathbb{R}^d\setminus B(A,|\ln R|^{-\delta})}\frac{1}{|u-z|^d}\le C\norm{g}_{L^1}\,R|\ln R|^{d\delta}
	\end{split}
	\end{equation*}
	where $[x,y]$ is the line segment joining $x$ and $y$, and where 
	we have used that $$|u-z|\ge |\ln R|^{-\delta} -\frac{R}{2}\ge \frac{|\ln R|^{-\delta}}{2}$$ in the considered supremum. Therefore we have obtained
	\begin{equation*}\begin{split}
	I_1&\le C\norm{g}_{L^1}R|\ln R|^{\beta}.
	\end{split}
	\end{equation*}

	\medskip

	For $I_3$ we apply H\"older's inequality for Orlicz spaces,
	\begin{equation*}
	I_3\le C\norm{g}_{L_{\phi_\alpha}}\norm{1_{B(A,R)}|K(x-z)-K(y-z)|}_{L_{\bar{\phi}_\alpha}}\le C\norm{g}_{L_{\phi_\alpha}}\norm{1_{B(0,3R/2)}K}_{L_{\bar{\phi}_\alpha}}
	\end{equation*}
	where we have used the fact that {$\bar{\phi}_\alpha$ is increasing}, that
	\begin{equation*}
	|K(x-z)-K(y-z)|\le |K(x-z)|+|K(y-z)|
	\end{equation*}
	and that $z\in B(A,R)$ implies that both $x-z$ and $y-z$ lie in $B(0,3R/2)$.
	
	Now we set $\lambda=R|\ln R|^{1/\alpha}$ and we consider the integral
	\begin{equation*}
	\int_{\R^d}\bar{\phi}_\alpha(1_{B(0,3R/2)}|K(z)|/\lambda)\,dz.
	\end{equation*}
	By \Cref{rem:bounding-orlicz-norms}, to show that $I_3\le C\norm{g}_{L_{\phi_\alpha}}\lambda$ it is sufficient to show that the 
	integral above is bounded by a constant.
	Furthermore, by \Cref{rem:equivalent-N-functions} using the fact that $|K(z)|/\lambda \ge 1>0$ on $B(0,3R/2)$  
	we may work with the asymptotic form \eqref{eq:asymptotic-of-bar-phi-alpha}. 
	
Thus, we estimate
	\begin{align*}
	\int_{|z|\le 3R/2}\frac{|K(z)|}{\lambda} \ln\left(\frac{|K(z)|}{\lambda}\right)^{1/\alpha}\,dz
	&=\frac{1}{\lambda}\int_{|z|\le 3R/2}|z|^{{1-d}}\ln\left(\frac{|z|^{1-d}}{\lambda}\right)^{1/\alpha}\,dz\\
	&\leq \frac{C}{\lambda}\int^{3R/2}_0|\ln r |^{1/\alpha}\,dr\end{align*}
	where we have used the inequality 
	$$0\leq \ln\left(\frac{|z|^{1-d}}{\lambda}\right)\leq (d-1)|\ln |z||+|\ln \lambda|\leq C|\ln |z||$$ 
	with this definition of $\lambda$. 
	
	Thus, noting that  for $r\leq 3R/2\leq 1/6$ we have
	$$ \frac{1}{\alpha}|\ln r |^{(1/\alpha)-1}\leq \frac{1}{2} |\ln r|^{1/\alpha},$$ so that 
	$$|\ln r|^{1/\alpha}\leq 2\left(|\ln r|^{1/\alpha}-\frac{1}{\alpha}|\ln r|^{(1/\alpha)-1}\right),$$
	 we obtain
\begin{align*}\int_{|z|\le 3R/2}\frac{|K(z)|}{\lambda} \ln\left(\frac{|K(z)|}{\lambda}\right)^{1/\alpha}\,dz
	&\le \frac{C}{\lambda}\int^{3R/2}_0\left(|\ln r|^{1/\alpha}-\frac{1}{\alpha}|\ln r|^{(1/\alpha)-1}\right)\,dr\\
	&=\frac{C}{\lambda}\left[r|\ln r|^{1/\alpha}\right]^{3R/2}_0\le \frac{C}{\lambda} R|\ln R|^{1/\alpha}.
	\end{align*}  Thus we have shown that
$$I_3\leq C \norm{g}_{L_{\phi_\alpha}} R|\ln R|^{1/\alpha}.$$

	Finally we bound $I_2$. In the same way as for $I_3$ we apply  H\"older's inequality for Orlicz spaces to obtain
	\begin{equation*}
	I_2\le C\norm{g}_{L_{\phi_\alpha}}\norm{1_{B(A,|\ln R|^{-\delta})\setminus B(A, R)}
	|K(x-z)-K(y-z)|}_{L_{\bar{\phi}_{\alpha}}}.
	\end{equation*}
	Applying the mean value theorem we obtain  for $z\in B(A,|\ln R|^{-\delta})\setminus B(A, R)$ 
	\begin{equation*}\begin{split}
	|K(x-z)-K(y-z)|&\le C\,R\sup_{u\in[x,y]}|u-z|^{-d}\\
	&\le C\,R\sup_{u\in[x,y]}\frac{1}{||z-A|-|u-A||^d}\\
	&\le C\,R|z-A|^{-d}
	\end{split}
	\end{equation*}
	where we have used that $|z-A|\ge R$ to obtain the final inequality. Hence, by a change of variables, and since $\bar{\phi}_\alpha$ is increasing, 
	to bound $I_2$ it is sufficient to obtain the bound 

	$$	{\norm{1_{B(0,|\ln R|^{-\delta})\setminus B(0,R)} R|z|^{-d}}_{L_{\bar{\phi}_\alpha}}\le C\,R|\ln R|^{\beta}.}$$ 

	Therefore setting $\lambda'=R|\ln R|^{\beta}$,	by \Cref{rem:bounding-orlicz-norms} it is enough to show that 
	\begin{equation*}
	\int_{R\le |z|\le |\ln R|^{-\delta}}\bar{\phi}_\alpha(R|z|^{-d}/\lambda')\,dz\le C.
	\end{equation*}
	Let $|z|\leq |\ln R|^{-\delta}$, then we have by definition of $\lambda'$
	$$\frac{R|z|^{-d}}{\lambda'}\ge \frac{R|\ln R|^{d\delta}}{\lambda'}=1,$$
so	by \Cref{rem:equivalent-N-functions} we may instead bound the asymptotic form \eqref{eq:asymptotic-of-bar-phi-alpha}. Therefore, we estimate
	\begin{align*}
	\int_{R\le |z|\le |\ln R|^{-\delta}}\frac{R|z|^{-d}}{\lambda'}\ln\left(\frac{R|z|^{-d}}{\lambda'}\right)^{1/\alpha}\,dz
	&=|\ln R|^{-\beta}\int_{R\le |z|\le |\ln R|^{-\delta}}|z|^{-d}\ln(|z|^{-d}|\ln R|^{-\beta})^{1/\alpha}\,dz\\
	&= C|\ln R|^{-\beta}
	\int^{|\ln R|^{-\delta}}_R r^{-1}|\ln (r^{-d}|\ln R|^{-\beta})|^{1/\alpha}\,dr
	.\end{align*}
	Since for $r\leq |\ln R|^{-\delta}$ we have  $$|\ln (r^{-d}|\ln R|^{-\beta})|=\ln (r^{-d}|\ln R|^{-\beta})=d|\ln r|-\beta \ln |\ln R|\leq d|\ln r|,$$
	we infer that
	\begin{align*}
	\int_{R\le |z|\le |\ln R|^{-\delta}}\frac{R|z|^{-d}}{\lambda'}\ln
	\left(\frac{R|z|^{-d}}{\lambda'}\right)^{1/\alpha}\,dz
	&\le C|\ln R|^{-\beta} \int^{|\ln R|^{-\delta}}_R r^{-1}|\ln r|^{1/\alpha}\,dr\\
	&\le C|\ln R|^{-\beta}
	\left[-|\ln r|^{(1/\alpha)+1}\right]^{|\ln R|^{-\delta}}_R\\
	&\leq C,
	\end{align*}as we wanted,	
	and hence we obtain
	$$I_2\le C\norm{g}_{L_{\phi_{\alpha}}(\mathbb{R}^d)} R|\ln R|^{\beta}.$$
	\medskip
	Finally, putting this all together, we conclude that
	\begin{align*}
	&\int_{\mathbb{R}^d} \left|K(x-z)-K(y-z)\right||g(z)|\,dz \\
	&\quad \le I_1+I_2+I_3\le C(\norm{g}_{L^1}R+\norm{g}_{L_{\phi_\alpha}}R|\ln R|^{\beta}+\norm{g}_{L_{\phi_\alpha}}R|\ln R|^{1/\alpha})
	\end{align*}
	which implies the claim of the lemma.
\end{proof}

\subsection{Lagrangian formulation of the Vlasov-Poisson system and the Wasserstein distance}

Let  $f\in C([0,T],\mathcal{M}_+(\mathbb{R}^d\times \mathbb{R}^d)- w^\ast)$ be a weak measure-valued solution of the Vlasov-Poisson 
system \eqref{eq:Vlasov-Poisson} on $[0,T]$ such that $\rho\in L^\infty([0,T],L^1\cap L^p(\mathbb{R}^d))$ for some $p>d$. By  potential estimates it is well-known that   $E\in L^\infty([0,T]\times \mathbb{R}^d)$. Moreover, by Cald\'eron-Zygmund inequality (see e.g. see \cite[Theo. 4.12]{javier}) $\nabla E\in L^\infty([0,T],L^p(\mathbb{R}^d))$. By the theory on transport equations (see \cite[Theo. III2]{dip-lions} or 
\cite[Theo. 5.7]{ambrosio-crippa} for more recent results on the theory), there exists a unique Lagrangian flow associated to $E$, namely a  map $(X,V)\in L_{\text{loc}}^1([0,T]\times \mathbb{R}^d\times\mathbb{R}^d; \mathbb{R}^d\times \mathbb{R}^d)$ such that for a.e. $(x,v)\in \mathbb{R}^d \times \mathbb{R}^d$, $t\mapsto (X,V)(t,x,v)$ is an absolutely continuous integral solution of the characteristic system of ODE
\begin{equation}\label{eq:characteristics}
\begin{cases}
\displaystyle \dot{X}(t,x,v)=V(t,x,v),\quad X(0,x,v)=x\\
\displaystyle \dot{V}(t,x,v)=E(t,X(t,x,v)),\quad V(0,x,v)=v.
\end{cases}\end{equation}
Moreover, we have the representation\footnote{The $\#$ notation means that $f(t)(B)=f_0\left(((X,V)(t,\cdot,\cdot)^{-1}(B)\right)$ for all Borel set $B\subset \mathbb{R}^d$.}
\begin{equation}\label{representation}\forall t\in [0,T],\quad f(t)=(X,V)(t)_\#f_0.\end{equation}

\medskip

Let $f_1,f_2$ be two weak solutions of the Vlasov-Poisson equation \eqref{eq:Vlasov-Poisson} as in Theorem \ref{thm:main}, then $f_j(t)=(X_j(t),V_j(t))_\#f_{j0}$ for $(X_j,V_j)(t,x,v)$ the solutions to the characteristic equations \eqref{eq:characteristics} associated to $E_j$.

\begin{remark}\label{rem:Holder-flow}
	In fact, under the assumptions of Theorem \ref{thm:main}, the 
	characteristic flows are H\"older continuous as functions of $(x,v)$. 
	This may be deduced from a similar Gr\"onwall type estimate to the proof of Lemma \ref{lem:estimate-on-D} below  
	using that $E_i$ satisfy a log${}^2$-Lipschitz bound of the form \eqref{eq:log2-lipschitz-assumption}. This will not be needed for the proof of Theorem \ref{thm:main}.
\end{remark}

\medskip

 Given a coupling $\pi_0\in \Pi(f_{10},f_{20})$ (as defined in \Cref{def:W-1}) we define the following quantities:
 \begin{equation}\label{def:distance}
\begin{split}
&\mathcal{X}(t)=\mathcal{X}_{\pi_0}(t)=\int_{\R^{2d}\times \R^{2d}} |X_1(t,x,v)-X_2(t,y,w)|\,d\pi_0(x,v,y,w),\\
&\mathcal{V}(t)=\mathcal{V}_{\pi_0}(t)=\int_{\R^{2d}\times \R^{2d}} |V_1(t,x,v)-V_2(t,y,w)|\,d\pi_0(x,v,y,w).\\
\end{split}
\end{equation}

\medskip

By \eqref{representation}, {the measure $\pi_t=\left((X_1(t),V_1(t));(X_2(t),V_2(t))\right)_\#\pi_0$ belongs to $\Pi(f_1(t),f_2(t))$. Therefore,
by the  \Cref{def:W-1} of the Wasserstein distance, we have
\begin{equation}\label{carac:wasserstein}
W_1(f_1(t),f_2(t))\leq \inf_{\pi_0 \in \Pi(f_{10},f_{20})} \left(\mathcal{X}_{\pi_0}(t)+\mathcal{V}_{\pi_0}(t)\right),\quad t\in [0,T].
\end{equation}
On the other hand, note the converse estimate:
\begin{equation}\label{carac:wasserstein-initial}
\begin{split} \inf_{\pi_0 \in \Pi(f_{10},f_{20})} \left(\mathcal{X}_{\pi_0}(0)+\mathcal{V}_{\pi_0}(0)\right)
& \leq \inf_{\pi_0 \in \Pi(f_{10},f_{20})} \int_{\R^{2d}\times \R^{2d}} (|x-y|+|v-w|)\,d\pi_0(x,v,y,w)\\
& \leq \inf_{\pi_0 \in \Pi(f_{10},f_{20})} \int_{\R^{2d}\times \R^{2d}} \sqrt{2} |(x,v)-(y,w)|\,d\pi_0(x,v,y,w)\\
&=\sqrt{2}W_1(f_{1,0},f_{2,0}).\end{split}
\end{equation}}

We emphasize that the quantity which would lead to \eqref{ineq:loeper} in \cite{loeper} is instead
$$\left\{
\int_{\R^{2d}\times \R^{2d}} \left(|X_1(t,x,v)-X_2(t,y,w)|^2+|V_1(t,x,v)-V_2(t,y,w)|^2\right)\,d\pi_0(x,v,y,w)\right\}^{1/2},$$ since it controls the Wasserstein distance $W_2(f_1(t),f_2(t))$.

\subsection{Proof of Theorem \ref{thm:main} completed}

We will prove \Cref{thm:main} proper with the following lemma which controls {the distance involving the spatial characteristics, namely the quantity 
$\mathcal{X}(t)$.}

We recall that $\beta=1+(1/\alpha)$.
For a given $0<A<1/9$, we define the function $G_\alpha(t)=G_\alpha(t;A)=G_{\alpha}(t;A,c)$ as the solution to
\begin{align}\label{eq:G}
G'_\alpha(t;A,c)&=-cG_\alpha(t;A,c)^{\beta/2},& G_\alpha(0;A,c)&=\ln(1/A)
\end{align}
for times $t\in[0,T^*]$ where $T^*=T^*(\alpha,A,c)$ is the maximal time such that $G_\alpha(\cdot;A,c)> \ln(9)$ on 
$[0,T^\ast)$ (we set $T^\ast=T$ if $T^\ast$ is larger than $T$). Note that $G_\alpha(\cdot;A,c)$ is decreasing and  is explicitly given by
\begin{equation}\label{eq:G-explicit}
G_\alpha(t;A,c)=\begin{cases}\displaystyle
|\ln A|\exp(-ct)&\text{when }\alpha=1,\\
\displaystyle \left(| \ln A|^{1/\gamma}-c\gamma^{-1}t\right)^{\gamma} & \text{when }\alpha\ne1,
\end{cases}
\end{equation}
where $\gamma$ is given by \eqref{eq:gamma}.
Moreover, we have
\begin{equation}\label{eq:Tast}
 T^\ast(\alpha, A,c)=\begin{cases}\displaystyle \frac{1}{c}\ln\left(\frac{|\ln A|}{\ln 9}\right)\quad \text{when }\alpha=1,\\ 
 \displaystyle \frac{\gamma}{c}\left(|\ln A|^{1/\gamma}-(\ln 9)^{1/\gamma}\right)\quad \text{when }\alpha \neq 1.\end{cases} 
\end{equation}

\begin{lemma}\label{lem:estimate-on-D}
	Let $\pi_0\in \Pi(f_{10},f_{20})$ and \eqref{eq:assumption-on-rhos} hold. Assume that $$0\leq  A:=(1+T)(\mathcal{X}(0)+\mathcal{V}(0))<\frac{1}{9}.$$ 
	Then the following estimate holds
	\begin{equation*}
	\mathcal{X}(t)\le \exp(-G_\alpha(t;A(t),c)), \quad  t\leq T^\ast(\alpha, A,c),
	\end{equation*}where $$A(t)=(1+t)(\mathcal{X}(0)+\mathcal{V}(0))\leq A$$ and 
	where $c>0$ is a constant depending only on $\alpha$ and the norms in \eqref{eq:assumption-on-rhos}. 
\end{lemma}
\begin{proof}
By integrating the characteristic ODEs \eqref{eq:characteristics} twice we have
	\begin{equation}\label{eq:integrated-twice-characteristics}
	\begin{aligned}
	&X_1(t,x,v)-X_2(t,y,w)\\
	&\qquad=x-y+(v-w)t+\int^t_0\int^s_0 \left[E_1(\tau,X_1(\tau,x,v))-E_2(\tau,X_2(\tau,y,w))\right]\,d\tau ds.
	\end{aligned}
	\end{equation}  
	Since $f_j(t)=(X_j(t),V_j(t))\#f_{j0}$, we can evaluate the fields $E_j$ as follows, where we omit the $\tau$ dependence for brevity:
	\begin{align*}
	&E_1(X_1(x,v))-E_2(X_2(y,w))\\
	&\quad=\int_{\R^{2d}} K(X_1(x,v)-X_1(x_0,v_0))f_{10}(x_0,v_0)\,dx_0dv_0\\
	&\qquad -\int_{\R^{2d}} K(X_2(x,v)-X_2(y_0,w_0))f_{20}(y_0,w_0)\,dy_0dw_0\\
	&\quad=\int_{\R^{2d}\times \R^{2d}} [K(X_1(x,v)-X_1(x_0,v_0))-K(X_2(x,v)-X_2(y_0,w_0))]\,d\pi_0(x_0,v_0,y_0,w_0)\\
	&\quad=\int_{\R^{2d}\times \R^{2d}} [K(X_1(x,v)-X_1(x_0,v_0))-K(X_2(x,v)-X_1(x_0,v_0))]\,d\pi_0(x_0,v_0,y_0,w_0)\\
	&\qquad+\int_{\R^{2d}\times \R^{2d}} [K(X_2(x,v)-X_1(x_0,v_0))-K(X_2(x,v)-X_2(y_0,w_0))]\,d\pi_0(x_0,v_0,y_0,w_0)\\
	&\quad=\int_{\R^{2d}\times \R^{2d}} [K(X_1(x,v)-z)-K(X_2(x,v)-z)]\rho_1(z)\,dz\\
	&\qquad+\int_{\R^{2d}\times \R^{2d}} [K(X_2(x,v)-X_1(x_0,v_0))-K(X_2(x,v)-X_2(y_0,w_0))]\,d\pi_0(x_0,v_0,y_0,w_0).
	\end{align*}
	Thus, by applying \Cref{lem:harmonic-analysis} we obtain the estimate
	\begin{align*}
	|E_1&(X_1(x,v))-E_2(X_2(y,w))|\le C\psi_\alpha(|X_1(x,v)-X_2(x,v)|)\\
	&+\int_{\R^{2d}\times \R^{2d}}| K(X_2(x,v)-X_1(x_0,v_0))-K(X_2(x,v)-X_2(y_0,w_0))|\,d\pi_0(x_0,v_0,y_0,w_0).
	\end{align*}
	It follows that
		\begin{align*}
&\int_{\R^{2d}\times \R^{2d}}|E_1(X_1(x,v))-E_2(X_2(y,w))|d\pi_0(x,v,y,w)\\
&\le C \int_{\R^{2d}\times \R^{2d}} \psi_\alpha(|X_1(x,v)-X_2(x,v)|)\,d\pi_0(x,v,y,w)\\
&+\int_{\R^{2d}\times \R^{2d}}\left(\int_{\R^{2d}\times \R^{2d}}|K(X_2(x,v)-X_1(x_0,v_0))-K(X_2(x,v)-X_2(y_0,w_0))|\,d\pi_0(x_0,v_0,y_0,w_0)\right)d\pi_0(x,v,y,w)\\
&\le C \int_{\R^{2d}\times \R^{2d}} \psi_\alpha(|X_1(x,v)-X_2(x,v)|)\,d\pi_0(x,v,y,w)\\
&+\int_{\R^{2d}\times \R^{2d}}\left(\int_{\R^{2d}\times \R^{2d}}|K(z-X_1(x_0,v_0))-K(z-X_2(y_0,w_0))|\rho_2(\tau,z)\,dz\right)d\pi_0(x_0,v_0,y_0,w_0)
	\end{align*}where we have exchanged
 the order of integration with $d\pi_0(x_0,v_0,y_0,w_0)$ and used that $f_1(\tau)=(X_1(\tau),V_1(\tau))\#f_{10}$ in the last inequality.
Therefore, by integrating \eqref{eq:integrated-twice-characteristics} against the measure $d\pi_0(x,v,y,w)$ we obtain	
	\begin{align*}
	\mathcal{X}(t)&\le \int_{\R^{2d}\times \R^{2d}} |x-y+t(v-w)|\,d\pi_0(x,v,y,w)\\
	&+\int_0^t\int_0^s \int_{\R^{2d}\times \R^{2d}}|E_1(\tau,X_1(\tau,x,v))-E_2(\tau,X_2(\tau,y,w))|d\pi_0(x,v,y,w)\,d\tau\,ds\\
	&\le [\mathcal{X}(0)+t\mathcal{V}(0)]+2C\int^t_0\int^s_0\int_{\R^{2d}\times \R^{2d}} \psi_\alpha(|X_1(\tau,x,v)-X_2(\tau,x,v)|)\,d\pi_0(x,v,y,w)\,d\tau\, ds\\
	\end{align*}
	where we have applied \Cref{lem:harmonic-analysis} (noting \Cref{rem:alpha-infinity-case} if $\alpha=\infty$) to find the second inequality.
	Using that $\psi_\alpha$ is concave we deduce that 
	\begin{align*}
	\mathcal{X}(t)
	&\le (\mathcal{X}(0)+\mathcal{V}(0)(1+t)+C_0\int^t_0\int^s_0\psi_\alpha(\mathcal{X}(\tau))\,d\tau \,ds
	\end{align*}
	for a constant $C_0$ depending only on $\alpha$ and the norms in \eqref{eq:assumption-on-rhos}.
	For a constant $c$ to be determined later on, let $T^\ast(\alpha, A,c)$ be the corresponding time defined by \eqref{eq:Tast}.
	Let $t_0\in [0,T^\ast(\alpha,A,c)]$ be fixed and set
	\begin{equation*}
	\mathcal{F}(t)=(\mathcal{X}(0)+\mathcal{V}(0)(1+t_0)+C_0\int^t_0\int^s_0\psi_\alpha(\mathcal{X}(\tau))\,d\tau\, ds\ge \mathcal{X}(t),\quad t\in [0,t_0].
	\end{equation*} Define $\varphi_\alpha(t)=\int^t_0\psi_\alpha(s)\,ds$ and note that 
$\varphi_\alpha(\tau)\le C \tau^2|\ln(
\tau)|^\beta$ for $\tau\le 1/9$ and  $\varphi(\tau)\leq C\tau$ for $\tau \geq 1/9$. Then it holds that
	\begin{equation*}
	\mathcal{F}(0)=(\mathcal{X}(0)+\mathcal{V}(0)(1+t_0), \quad \mathcal{F}'(0)=0,\quad \mathcal{F}''(t)
	=C_0\psi_\alpha(\mathcal{X}(t))\le C_0\psi_\alpha(\mathcal{F}(t))=C_0\varphi_\alpha'(\mathcal{F}(t))
	\end{equation*}
	and $\mathcal{F}'(t)\ge0$. Thus 
	\begin{equation*}
	[(\mathcal{F}'(t))^2]'=2\mathcal{F}''(t)\mathcal{F}'(t)\le 2C_0\varphi_\alpha'(\mathcal{F}(t))\mathcal{F}'(t)=2C_0[\varphi_\alpha(\mathcal{F}(t))]'
	\end{equation*}
	and by integrating we deduce that
	\begin{equation*}
	\mathcal{F}'(t)\le \sqrt{2C_0\,\varphi_\alpha(\mathcal{F}(t))},\quad \text{for }t\leq t_0,
	\end{equation*} 
	which by definition of $\varphi_\alpha$ implies \begin{equation}\label{eq:differential-inequality-for-F}
	\mathcal{F}'(t)\le C_1\,\mathcal{F}(t)|\ln \mathcal{F}(t)|^{\beta/2},\quad \text{for }\mathcal{F}(t)\leq \frac{1}{9},
	\end{equation} 
	where $C_1$ depends only on $\alpha$ and the norms in \eqref{eq:assumption-on-rhos}. Now let
	 $y(t)$ be the solution to 
	\begin{equation*}
	y'(t)= C_1 y(t)|\ln y(t)|^{\beta/2}, \quad  y(0)=\mathcal{F}(0)=A(t_0),
	\end{equation*}
for $t\in [0,T^\ast(\alpha,A,c)]$. In view of the definition \eqref{eq:Tast}, since $A(t_0)\leq A$ we have $T^\ast(\alpha,A(t_0),C_1)\geq T^\ast(\alpha, A,C_1)$ so that
 $y\leq 1/9$ on $[0,T^\ast(\alpha,A,C_1)]$. 
Then \eqref{eq:differential-inequality-for-F} obeys $\mathcal{F}(t)\leq y(t)\leq 1/9$ on its domain of definition. By applying the change of variables $y=e^{-G}$  
 we deduce that $G'=-C_1G^{\beta/2}$ and that therefore
 $$\mathcal{F}(t)\leq y(t)=\exp(-G_\alpha(t,A(t_0),C_1)),\quad t\in [0,t_0],$$
and as $t_0$ was arbitrary the proof of the lemma is complete by setting $c=C_1$.
	\end{proof}

Using this lemma we are now able to prove the main result \Cref{thm:main}.
\begin{proof}[Proof of \Cref{thm:main}]
	By integrating the characteristic equation \eqref{eq:characteristics} once we obtain
	\begin{equation}
	V_1(t,x,v)-V_2(t,y,w)=v-w+\int^t_0 \left[E_1(s,X_1(s,x,v))-E_2(s,X_2(s,y,w))\right]\,ds.
	\end{equation}
	Letting $\pi_0\in \Pi(f_{10},f_{20})$ be arbitrary, in the same way as in the proof of \Cref{lem:estimate-on-D} we find that
	\begin{equation*}
	\mathcal{V}(t)\le \mathcal{V}(0)+C\int^t_0\psi_\alpha(\mathcal{X}(s))\,ds.
	\end{equation*}
	By \eqref{carac:wasserstein-initial}, we may consider only couplings $\pi_0$ such that $\mathcal{X}(0)+\mathcal{V}(0)\leq 2 W_1(f_{1}(0),f_{2}(0))$ and, therefore,
	by assumption on $W_1(f_{1}(0),f_{2}(0))$ in \Cref{thm:main} 
$$A:=(\mathcal{X}(0)+\mathcal{V}(0))(1+T)< \frac{1}{9}.$$ So we also have $\mathcal{X}(t)\le \exp(-G_\alpha(t,A(t),c))$ with 
$A(t)=(\mathcal{X}(0)+\mathcal{V}(0))(1+t)$ by \Cref{lem:estimate-on-D} and for $t\leq T^\ast(\alpha,A,c)$. Note that by definition \eqref{eq:Tast} of the time $T^\ast$, since
$W_1(f_{1}(0),f_{2}(0))(1+T)\leq A$ we have 
 $$T^\ast(\alpha, A, c)\geq T^\ast:= T^\ast(\alpha, W_1(f_{1}(0),f_{2}(0))(1+T),c).$$ Thus all the subsequent estimates hold for $t\in [0,T^\ast]$.
Thus, since $\psi_\alpha$ is an increasing function, we obtain, dropping the $\alpha,A$ and $c$ in $G$  for brevity,
	\begin{align*}
	\int^t_0\psi_\alpha(\mathcal{X}(s))\,ds&\le \int^t_0\psi_\alpha(\exp(-G(s)))\,ds=\int^{t}_0\exp(-G(s))G(s)^{\beta}\,ds\\
	&\le \int^{t}_0\exp(-G(t))G(0)^{\beta}\,ds=t\exp(-G(t))G(0)^{\beta}\\
	&=t\exp(-G(t))|\ln A(0)|^\beta,
	\end{align*}
	 where we have used that $s\mapsto G(s)=G_\alpha(s,A(s),c)$ is a decreasing function of $s$.
	
	Combining the estimates for $\mathcal{X}$ and $\mathcal{V}$ we have
	\begin{align*}
	\mathcal{X}(t)&+\mathcal{V}(t)\\&\le \exp(-G_\alpha(t;A(t),c))+\mathcal{V}(0)+Ct\exp(-G_\alpha(t;A(t),c))|\ln A(0)|^\beta\\
	&\leq \exp(-G_\alpha(t;A(t),c))+A(t)+Ct\exp(-G_\alpha(t;A(t),c))|\ln A(0)|^\beta\\
	&= \exp(-G_\alpha(t;A(t),c))+\exp(-G_\alpha(0,A(t),c))+Ct\exp(-G_\alpha(t;A(t),c))|\ln A(0)|^\beta\\
	&\leq 2 \exp(-G_\alpha(t;A(t),c))+Ct\exp(-G_\alpha(t;A(t),c))|\ln A(0)|^\beta
	\end{align*}
	where we have used that $t\mapsto G_\alpha(t,\overline{A},c)$ is decreasing for fixed $\overline{A}$ in the last line.
	Thus for $t\in [0,T^\ast]$ we obtain
	\begin{align*}
	\mathcal{X}(t)+\mathcal{V}(t)&	
	\leq \exp(-G_\alpha(t;A(t),c))
	(2+C\,t|\ln A(0)|^\beta).
	\end{align*}
	We set $B=W_1(f_{1}(0),f_{2}(0))$, so that $B\leq A(0)\leq 2 B$.
	By taking the infimum over couplings $\pi_0$ (recall \eqref{carac:wasserstein} and \eqref{carac:wasserstein-initial}) we obtain
	\begin{equation*}
	W_1(f_1(t),f_2(t))\le \exp(-G_\alpha(t;2(1+t)B))(2+C\, t|\ln B|^\beta).
	\end{equation*}

	Now suppose $\alpha=1$, then by the explicit formula \eqref{eq:G-explicit} we have (recalling that $\beta=2$ in this case)
	\begin{align*}
	W_1(f_1(t),f_2(t))&\le \exp(\ln(2(1+t)B)\exp(-ct))(2+C\,t\,|\ln B|^2)\\
	&\le (2(1+t)B)^{\exp(-ct)}(2+C\,t|\ln B|^2)\\
	&\le CB^{\exp(-ct)}(1+t|\ln B|^2)
	\end{align*}
	where we have used that $(2(1+t))^{\exp(-ct)}$ is bounded by a constant uniformly over $t\in[0,\infty)$.
	
	Suppose instead that $\alpha\ne1$, then 
	\begin{align*}
		W_1(f_1(t),f_2(t))&\le \exp(-G_\alpha(t;A(t),c))(2+C\,t|\ln B|^\beta)\\
		&\le \exp(\gamma^{-1}\ln(2B(1+t))+Ct^\gamma)(2+C\,t|\ln B|^\beta)\\
		&\le (2B(1+t))^{1/\gamma}\exp(Ct^\gamma)(2+C\,t |\ln B|^\beta)\\
		&\le CB^{1/\gamma}\exp(Ct^\gamma)(1+t|\ln B |^\beta)
	\end{align*}
	where on the last line we have used that $e^{Ct^\gamma}t^\delta\le Ce^{C't^\gamma}$ for a larger constant $C'>C$, and on the second line we have used the lower bound
	\begin{equation}
	G_\alpha(t;A,c)\ge\gamma^{-1}\ln(1/A)-Ct^\gamma 
	\end{equation}
	for $\alpha\ne1$, which we will now prove. Indeed, from \eqref{eq:G-explicit} we use convexity (noting $\gamma\ge2$) to obtain
	\begin{equation*}
	G(t)\ge G(0)-tG'(0)=G(0)-CtG(0)^{\beta/2},
	\end{equation*}
	and the desired bound follows from an application of Young's inequality, i.e. $G(0)^{\beta/2}t\le \gamma^{-1}t^\gamma+(\beta/2) G(0)$.

Finally, we infer the lower bound for $T^\ast=T^\ast(\alpha, B(1+T),c)$ as follows: since $(1+T)^{1+\varepsilon} B\leq 1$, 
	we have $$|\ln B(1+T)|\geq \frac{\varepsilon}{1+\varepsilon} |\ln B|.$$ So in view of \eqref{eq:Tast}, we have for $\alpha=1$
\begin{equation*}
 T^\ast\geq \frac{1}{c}\left(\ln |\ln B|
 +\ln \left(\frac{\varepsilon}{1+\varepsilon}\right)-\ln \ln 9\right)\geq C' \ln |\ln B|-C'^{-1}, 
\end{equation*}
and for $\alpha>1$
\begin{equation*}
 T^\ast\geq  \frac{\gamma}{c}\left(\left(\frac{\varepsilon}{1+\varepsilon}\right)^{1/\gamma}
 |\ln B|^{1/\gamma}-(\ln 9)^{1/\gamma}\right)\geq C'\gamma |\ln B|^{1/\gamma}-C'^{-1}
\end{equation*}
	for sufficiently large constant $C'$.
\end{proof}

\subsection{Proof of \Cref{thm:log2-lipschitz}}\label{subsec:proof-of-log2-lipschitz}

To prove \Cref{thm:log2-lipschitz} we note that we have the following result, analogous to \Cref{lem:harmonic-analysis}. As its proof is immediate we omit it.
\begin{lemma}
Let $K$ be bounded and satisfy \eqref{eq:log2-lipschitz-assumption}, then for any $\mu\in \mathcal{M}_+(\mathbb{R}^d)$ with mass $m$, we have the inequality
\begin{equation*}
\int_{\mathbb{R}^d}|K(x-z)-K(y-z)|d\mu(z)\le Cm\psi_{1}(|x-y|)
\end{equation*}
where $C$ is the constant in \eqref{eq:log2-lipschitz-assumption} and $\psi_1$ is defined by \eqref{eq:psi-alpha}.
\end{lemma}{Furthermore, we note that due to this lemma the vector fields $E_i$ are log${}^2$-Lipschitz, and as noted in \Cref{rem:Holder-flow} this is enough to define the characteristic ODEs.} The proof of \Cref{thm:log2-lipschitz} is now entirely analogous to the proof of \Cref{thm:main} for $\alpha=1$, replacing \Cref{lem:harmonic-analysis} with this lemma. Thus we leave it to the reader.

\section{Proof of \texorpdfstring{\Cref{prop:propagation-of-rho-bound}}{Proposition 1.1}}\label{sec:3}
We first show that it is sufficient to propagate the exponential velocity moment.
\begin{lemma}\label{lem:exp-rho-bounded-by-velocity-moment}
	Let $f\in L^\infty(\mathbb{R}^d\times\mathbb{R}^d)$ and $c>0$, then
	\begin{equation*}
	\int_{\R^d} \exp\left(\left|\int_{\R^d} f(x,v)\,dv\right|^\alpha/\lambda\right)\,dx\le C\int_{\R^d\times \R^d} f(x,v)e^{c\<v^{d\alpha}}\,dxdv
	\end{equation*} 
	for constants $C,\lambda$ depending only upon $\norm{f}_{L^\infty}$ and $c$.
\end{lemma}
\begin{proof}
	We apply the usual `interpolation' method: let
	\begin{equation*}
	M(x)=\int_{\R^d} f(x,v)e^{c\<v^{d\alpha}}\,dv,
	\end{equation*}
	then for each $x$ we have
	\begin{equation*}
	\rho(x):=\int_{\R^d} f(x,v)\,dv\le \int_{|v|\ge R} f(x,v)\,dv+C\norm{f}_{L^\infty}R^d\le e^{-cR^{d\alpha}}M(x)+CR^d,
	\end{equation*}
	by Markov's inequality. We now choose $R=R(x)=c^{-1/(\alpha d)}\ln(1\vee M(x))^{1/(d\alpha)}$ which gives
	\begin{equation*}
	\rho(x)\le 1+C\ln(1 \vee M(x))^{1/\alpha}.
	\end{equation*}
	Thus,
	\begin{align*}
	\int_{\R^d} \exp(|\rho(x)|^\alpha/\lambda)\,dx&\le \int_{\R^d} \exp(|1+C\ln(1\vee M(x))^{1/\alpha}|^\alpha/\lambda)\,dx\\
	&\le \int_{\R^d}\exp((C/\lambda)(1+\ln(1\vee M(x))))\,dx
	\end{align*}
	and choosing $\lambda=C$ we have
	\begin{equation*}
	\int_{\R^d} \exp(|\rho(x)|^\alpha/\lambda)\,dx\le C\int_{\R^d} M(x)\vee 1\,dx\le C\int_{\R^d} f(x,v)e^{c\<v^{d\alpha}}\,dxdv.\qedhere 
	\end{equation*}
\end{proof}
We now prove that the exponential moment is propagated. Since $f_0$ has finite velocity moments of order larger than $d^2-d$, the solution provided by 
\cite[Theo. 1]{lions-perthame} has bounded velocity moments of order larger than $d^2-d$ on $[0,T]$. By \cite[Cor. 2]{lions-perthame} it follows that
\begin{equation}\label{eq:E-in-L-infty}
\sup_{t\in[0,T]}\norm{E(t)}_{L^\infty}\le C<\infty 
\end{equation}
for any finite $T$. 
\begin{lemma}\label{lem:differential-inequality-for-M}
	Let $f\in L^\infty(\mathbb{R}^{d}\times \mathbb{R}^{d})$, $\alpha\ge1$ and $c>0$. Define 
	\begin{equation*}
	M(t)=\int_{\R^d\times \R^d} f(t,x,v)e^{1+c\<v^{d\alpha}}\,dx\,dv.
	\end{equation*}
	Then we have the differential inequality along the Vlasov-Poisson flow
	\begin{equation*}
	\frac{dM}{dt}\le C \big[1+\big(\ln M(t)\big)^{1-\tfrac{1}{d\alpha}}M(t)\big]
	\end{equation*}
	for a constant $C$ depending only upon $c,\alpha$ and $\norm{f}_{L^\infty}$.
\end{lemma}
\begin{proof}
	We directly compute, using the weak formulation of the Vlasov-Poisson equation
	\begin{align*}
	\frac{dM}{dt}&=-\int_{\R^d\times \R^d} cE(t,x)\cdot\nabla_v[\<v^{d\alpha}]f(t,x,v)e^{1+c\<v^{d\alpha}}\,dx\,dv\\
	&\le C \int_{\R^d\times \R^d} |E(t,x)|\<v^{d\alpha-1}f(t,x,v)e^{1+c\<v^{d\alpha}}\,dx\,dv\\
	&\le C\norm{E(t)}_{L^\infty}\int_{\R^d\times \R^d} \ln(e^{1+c\<v^{d\alpha}})^{1-\tfrac{1}{d\alpha}} e^{1+c\<v^{d\alpha}}f(t,x,v)\,dx\,dv.
	\end{align*}
	The claim of the lemma now follows from \eqref{eq:E-in-L-infty} and Jensen's inequality, using the convexity of $\tau\mapsto \tau\ln(\tau)^{1-\tfrac{1}{d\alpha}}$ on $\tau\in(e,\infty)$.
\end{proof}

\begin{proof}[Proof of \Cref{prop:propagation-of-rho-bound}]
	By using \Cref{lem:differential-inequality-for-M} and solving the resulting differential inequality we deduce that
	\begin{equation*}
	\sup_{t\in[0,T]}\int_{\R^d\times \R^d} f(t,x,v)e^{c\<v^{d\alpha}}\,dx\,dv\le C<\infty.
	\end{equation*} 
	The claim of the proposition now follows from an application of \Cref{lem:exp-rho-bounded-by-velocity-moment}.
\end{proof}

\medskip

\medskip

\textbf{Acknowledgments} The first author (T.H.) was supported during the preparation of this work by the UK Engineering and Physical Sciences Research Council (EPSRC) grant EP/H023348/1 for the University of Cambridge Centre for Doctoral Training, the Cambridge Centre for Analysis. The second author (E.M.) is or was partly supported during the preparation of this work by the ANR projects INFAMIE ANR-15-CE40-0,  SchEq
ANR-12-JS-0005-01 and GEODISP ANR-12-BS01-0015-01.

\end{document}